\documentclass[a4paper]{amsart}
\usepackage{txfonts, amsmath,amstext,amsthm,amscd,amsopn,verbatim,amssymb, amsfonts}
\usepackage{fullpage}

\usepackage[makeroom]{cancel}

\usepackage{float}
\restylefloat{table}

\usepackage[bbgreekl]{mathbbol}
\usepackage{tikz}
\usetikzlibrary{matrix}
\usetikzlibrary{shapes}
\usetikzlibrary{arrows}
\usetikzlibrary{calc,3d}
\usetikzlibrary{decorations,decorations.pathmorphing,decorations.pathreplacing,decorations.markings}
\usetikzlibrary{through}

\tikzset{ext/.style={circle, draw,inner sep=1pt},int/.style={circle,draw,fill,inner sep=1.4pt},nil/.style={inner sep=1pt}}
\tikzset{cy/.style={circle,draw,fill,inner sep=2pt},scy/.style={circle,draw,inner sep=2pt},scyx/.style={draw,cross out,inner sep=2pt},scyt/.style={draw,regular polygon,regular polygon sides=3,inner sep=0.95pt}}
\tikzset{exte/.style={circle, draw,inner sep=3pt},inte/.style={circle,draw,fill,inner sep=3pt}}
\tikzset{diagram/.style={matrix of math nodes, row sep=3em, column sep=2.5em, text height=1.5ex, text depth=0.25ex}}
\tikzset{diagram2/.style={matrix of math nodes, row sep=0.5em, column sep=0.5em, text height=1.5ex, text depth=0.25ex}}

\tikzset{
  rightblue/.style={
    decoration={markings,mark=at position .8 with {\arrow[scale=1.2,blue]{latex}}},
    postaction={decorate},
    shorten >=0.4pt}}
\tikzset{
  leftblue/.style={
    decoration={markings,mark=at position .55 with {\arrowreversed[scale=1.2,blue]{latex}}},
    postaction={decorate},
    shorten >=0.4pt}}
\tikzset{
  rightred/.style={
    decoration={markings,mark=at position .45 with {\arrow[scale=1.2,red]{latex}}},
    postaction={decorate},
    shorten >=0.4pt}}
\tikzset{
  leftred/.style={
    decoration={markings,mark=at position .2 with {\arrowreversed[scale=1.2,red]{latex}}},
    postaction={decorate},
    shorten >=0.4pt}}

\newcommand{\mul}{{
\begin{tikzpicture}[baseline=-.55ex,scale=.2]
 \node[circle,draw,fill,inner sep=.5pt] (a) at (0,0) {};
 \node[circle,draw,fill,inner sep=.5pt] (b) at (1,0) {};
 \draw (a) edge[bend left=25] (b);
 \draw (a) edge[bend right=25] (b);
\end{tikzpicture}}}

\newcommand{\notadp}{{
\begin{tikzpicture}[baseline=-.55ex,scale=.2, every loop/.style={}]
 \node[circle,draw,fill,inner sep=.5pt] (a) at (0,0) {};
 \draw (a) edge[loop] (a);
 \draw (-.2,-.2) -- (.2,.5);
\end{tikzpicture}}}

\newcommand{\mxto}[1]{{\overset{#1}{\longmapsto}}}
\newcommand{\tto}[1]{{\overset{#1}{\longrightarrow}}}

\theoremstyle{plain}
  \newtheorem{thm}{Theorem}
  \newtheorem{defi}[thm]{Definition}
  \newtheorem{prop}[thm]{Proposition}

  \newtheorem{lemma}[thm]{Lemma}
\theoremstyle{definition}
  
  \newtheorem*{rem}{Remark}

\newcommand{\K}{{\mathbb{K}}}
\newcommand{\Z}{{\mathbb{Z}}}

\newcommand{\GS}{\mathrm{GS}}

\newcommand{\GC}{\mathrm{GC}}
\newcommand{\fGC}{\mathrm{fGC}}
\newcommand{\sGC}{\mathrm{sGC}}
\newcommand{\fGCc}{\mathrm{fGCc}}

\newcommand{\mV}{\mathrm{V}}
\newcommand{\mE}{\mathrm{E}}
\newcommand{\mB}{\mathrm{B}}
\newcommand{\mO}{\mathrm{O}}

\DeclareMathOperator{\sgn}{sgn}

\newcommand{\grac}{\mathrm{grac}}

\begin{document}
\title{Multi-directed graph complexes and quasi-isomorphisms between them I: oriented graphs}

\author{Marko \v Zivkovi\' c}
\address{Mathematics Research Unit\\ University of Luxembourg\\ Grand Duchy of Luxembourg}
\email{marko.zivkovic@uni.lu}



\begin{abstract}
We construct a direct quasi-isomorphism from Kontsevich's graph complex $\GC_n$ to the oriented graph complex $\mO\GC_{n+1}$, thus providing an alternative proof that the two complexes are quasi-isomorphic.
Moreover, the result is extended to the sequence of multi-oriented graph complexes, where $\GC_n$ and $\mO\GC_{n+1}$ are the first two members.
These complexes play a key role in the deformation theory of multi-oriented props recently invented by Sergei Merkulov.
\end{abstract}

\maketitle

\section{Introduction}

Generally speaking, graph complexes are graded vector spaces of formal linear combinations of isomorphism classes of some kind of graphs. Each of graph complexes play a certain role in a subfield of homological algebra or algebraic topology. They have an elementary and simple combinatorial definition, yet we know very little about what their cohomology actually is.

In this paper we study a sequence of graph complexes $\mO_k\GC_n$, called $k$-oriented graph complexes, for $k\geq 0$. The complex $\mO_0\GC_n$ equals the well known Kontsevich's graph complex $\GC_n$ introduced by M.\ Kontsevich in \cite{Kont1}, \cite{Kont2}. Bigger number $k$ introduces $k$ different kinds (colours) of orientation on edges. See Section \ref{s:gc} below for the strict definition.

$1$-oriented graph complex, or simply oriented graph complex $\mO_1\GC_n$ is first introduced (to the knowledge of the author) by Merkulov. In \cite{oriented} Willwacher showed that it is quasi-isomorphic to Kontsevich's graph complex $\GC_{n-1}$. The same result is the consequence of the broader theory of Merkulov and Willwacher, \cite[6.3.8.]{MW}.
The purpose of this paper is to present the third, more direct proof of the same result, and to extend it to $k$-oriented graph complexes, as stated in the following theorem.

\begin{thm}
\label{thm:main}
For every $k\geq 0$ there is a quasi-isomorphism $\mO_k\GC_n\rightarrow\mO_{k+1}\GC_{n+1}$.
\end{thm}

Multi-oriented graph complexes play a key role in the deformation theory of multi-oriented props recently invented by Sergei Merkulov in \cite{Merk1} and to appear in \cite{Merk2}. Similarly to multi-oriented graphs, multi-oriented props have multiple directions on each edge, without loops. Props naturally have one basic direction that goes from the inputs to the outputs, and that direction is clearly without loops, see e.g.\ \cite{MV}. In his papers Merkulov gives a meaning to the extra directions, and provides interesting applications and representations of multi-oriented props.

Techniques used in this paper have recently been developed by the author to get new results about sourced graph complexes in \cite{Multi2}.

\subsection{Structure of the paper}
In Section \ref{s:gc} we define graph complexes needed in the paper. Sections \ref{s:sc} and \ref{s:sogc} introduce some sub-complexes that are more convenient to work with. Finally, in Section \ref{s:qi} we construct the quasi-isomorphism and prove Theorem \ref{thm:main}.

\subsection*{Acknowledgements}

I am very grateful to Sergei Merkulov for providing the motivation for this result, especially to extend it for $k>1$. I thank Thomas Willwacher for fruitful discussions.

\section{Graph complexes}
\label{s:gc}

In this section we define oriented graph complex $\mO_k\fGC_n$ for $n\in\Z$, $k\geq 0$, called \emph{$k$-oriented graph complex}.

Particularly $\mO_0\fGC_n$ is the Kontsevich's graph complex $\fGC_n$ defined for example in \cite{grt}, cf.\ \cite{eulerchar}. Strictly speaking, the complex we define here is the dual of the one in those papers, but it does not change the homology. Since the complex is defined as a formal vector space over the bases consisting graphs, the dual can be identified with the complex as the vector space. The real difference is in the differential.

We will work over a field $\K$ of characteristic zero. All vector spaces and differential graded vector spaces are assumed to be $\K$-vector spaces.

\subsection{Graphs}

\begin{defi}
Let $v>0$, $e\geq 0$ and $k\geq 0$ be integers. Let $V:=\{0,2,\dots,v-1\}$ be set of vertices, $E:=\{1,2,\dots,e\}$ set of edges and $K:=\{1,2,\dots,k\}$ set of ``colors''.

A \emph{graph} $\Gamma$ with $v$ vertices and $e$ edges is a map $\Gamma=(\Gamma_-,\Gamma_+):E\rightarrow V^2$ such that $\Gamma_-(a)\neq\Gamma_+(a)$ for every $a\in E$.

A \emph{$k$-oriented graph} is the graph $\Gamma$ together with maps $o_c:E\rightarrow \{+,-\}$ for $c\in K$ such that for every $c$ there is no cycle $a_0,a_1,\dots, a_i=a_0$ such that $\Gamma_{o_c(a_j)}(a_j)=\Gamma_{-o_c(a_{j+1})}(a_{j+1})$ for every $j=0,\dots,i-1$.
\end{defi}

The orientation of an edge $a$ from $\Gamma_-(a)$ to $\Gamma_+(a)$ is called the \emph{intrinsic} orientation of the edge $a$. For a color $c\in K$ the orientation $o_c(a)$ is called the \emph{orientation of the color $c$} of the edge $a$ and goes from $\Gamma_{-o_c(a)}(a)$ to $\Gamma_{o_c(a)}(a)$. The condition in the definition of the oriented graph says that there is no oriented cycle in the graph in any color. Some examples of the graphs are drawn in Figure \ref{fig:exGraphs}. Note that by our definition a graph has distinguishable vertices and edges. No tadpoles (i.e.\ edges $a$ such that $\Gamma_-(a)=\Gamma_+(a)$) are allowed.

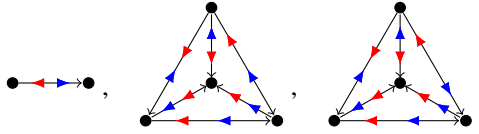
\begin{figure}[h]
\centering
$$
\begin{tikzpicture}[baseline=-1ex]
\node[int] (a) at (0,0) {};
\node[int] (b) at (1,0) {};
\draw (a) edge[->,rightblue,leftred] (b);
\end{tikzpicture}
\;,\quad
\begin{tikzpicture}[baseline=-1ex]
\node[int] (a) at (0,0) {};
\node[int] (b) at (90:1) {};
\node[int] (c) at (210:1) {};
\node[int] (d) at (330:1) {};
\draw (a) edge[<-,rightblue,leftred] (b);
\draw (a) edge[<-,rightblue,rightred] (c);
\draw (a) edge[<-,leftblue,leftred] (d);
\draw (b) edge[->,leftblue,rightred] (c);
\draw (b) edge[->,leftblue,leftred] (d);
\draw (c) edge[->,leftblue,leftred] (d);
\end{tikzpicture}
\;,\quad
\begin{tikzpicture}[baseline=-1ex]
\node[int] (a) at (0,0) {};
\node[int] (b) at (90:1) {};
\node[int] (c) at (210:1) {};
\node[int] (d) at (330:1) {};
\draw (a) edge[<-,rightblue,leftred] (b);
\draw (a) edge[<-,rightblue,rightred] (c);
\draw (a) edge[<-,leftblue,leftred] (d);
\draw (b) edge[->,leftblue,rightred] (c);
\draw (b) edge[->,rightblue,leftred] (d);
\draw (c) edge[->,leftblue,leftred] (d);
\end{tikzpicture}
\;.
$$

\caption{\label{fig:exGraphs}
Example of 2-oriented graphs. The intrinsic orientation is depicted by the simple black arrow, while the colored orientations are depicted by thick red and blue arrows. The last diagram does not represent an oriented graph because it contains a blue cycle.}
\end{figure}

\begin{defi}
For $a\in E$ we say that vertices $\Gamma_-(a)$ and $\Gamma_+(a)$ are connected. We extend the notion of being connected by transitivity, such that it is a relation of equivalence. Equivalence classes are called \emph{connected components}. A graph is \emph{connected} if it has one connected component and \emph{disconnected} if it has more than one connected component.
\end{defi}

Let
\begin{equation}
\mO_k\bar\mV_v\bar\mE_e\grac
\end{equation}
be the set of all connected $k$-oriented graphs with $v$ vertices and $e$ edges.

\subsection{Group actions on the set of graphs}

There is a natural action of the symmetric group $S_v$ on $\mO_k\bar\mV_v\bar\mE_e\grac$ that permutes vertices:
\begin{equation}
(\sigma\Gamma)(a)=(\sigma(\Gamma_-(a)),\sigma(\Gamma_+(a)))
\end{equation}
for $\sigma\in S_v$ and $a\in E$.

Similarly, there is an action of the symmetric group $S_e$ on $\mO_k\bar\mV_v\bar\mE_e\grac$ that permutes edges:
\begin{equation}
(\sigma\Gamma)(a)=\Gamma(\sigma^{-1}(a)),\qquad
(\sigma o_c)(a)=o_c(\sigma^{-1}(a))
\end{equation}
for $\sigma\in S_e$, $a\in E$ and $c\in K$.

Finally, for every edge $a\in E$ there is an actions of the symmetric group $S_2$ on $\mO_k\bar\mV_v\bar\mE_e\grac$ that reverses the intrinsic orientation of the edge $a$:
\begin{equation}
(\sigma\Gamma)(a)=(\sigma_+(a),\sigma_-(a)),\qquad
(\sigma o_c)(a)=-o_c(a)
\end{equation}
for $c\in K$ where $\sigma\in S_2$ is the non-trivial element. Note that the colored orientation is always preserved under this action relative to actual vertices. By this definition, the condition of having colored cycle is preserved.

All this actions together define the action of the group $S_v\times \left(S_e\ltimes S_2^{\times e}\right)$ on $\mO_k\bar\mV_v\bar\mE_e\grac$.

\subsection{Graded graph space}

\begin{defi}
Let $n$ be integer. Then
\begin{equation}
\mO_k\bar\mV_v\bar\mE_e\GS:=\langle\mO_k\bar\mV_v\bar\mE_e\grac\rangle[(v-1)n+(1-n)e].
\end{equation}
is the vector space of formal linear combinations of graphs. It is a graded vector space with non-zero term only in degree $d=(v-1)n+(1-n)e$.
\end{defi}

The action of the group $S_v\times \left(S_e\ltimes S_2^{\times e}\right)$ is by linearity extended to the space $\mO_k\bar\mV_v\bar\mE_e\GS$.

Let $\sgn_v$, $\sgn_e$ and $\sgn_2$ be one-dimensional representations of $S_v$, respectively $S_e$, respectively $S_2$, where the odd permutation reverses the sign. They can be considered as representations of the whole product $S_v\times \left(S_e\ltimes S_2^{\times e}\right)$.

\begin{defi}
For $k\geq 0$ the \emph{full $k$-oriented graph complex} is
\begin{equation}
\label{def:fGCi}
\mO_k\fGCc_n:=\left\{
\begin{array}{ll}
\displaystyle\bigoplus_{v,e}\left(\mO_k\bar\mV_v\bar\mE_e\GS\otimes\sgn_e\right)_{S_v\times \left(S_e\ltimes S_2^{e}\right)}
\qquad&\text{for $n$ even,}\\
\displaystyle\bigoplus_{v,e}\left(\mO_k\bar\mV_v\bar\mE_e\GS\otimes\sgn_v\otimes\sgn_2^e\right)_{S_v\times \left(S_e\ltimes S_2^{e}\right)}
\qquad&\text{for $n$ odd.}
\end{array}
\right.
\end{equation}
\end{defi}

The group in the subscript means taking coinvariants of the group action. The effect of tensoring with one-dimensional sign representation for $n$ even is that switching edges turns graph to its negative, while vertices and intrinsic orientation of edges are indistinguishable. We say that edges are odd, while vertices and intrinsic edge orientations are even. For $n$ odd edges are indistinguishable and switching vertices and intrinsic edge orientations in $\mO_k\fGCc_n$ turns graph to its negative. We say that vertices and intrinsic edge orientations are odd, while edges are even.

With abuse of terminology, the isomorphism class of a graph will also be called a graph. For distinguishing, any linear combination of graphs will not be called a graph.

Graphs will be drawn without mentioning number of vertices or edges, or without intrinsic orientation on edges, if the element is even in the complex. We still do the same if the element is odd, and the sign of the graph is not important. The colored orientation is always drawn since it is the essential data of the graph.

\begin{rem}
For $k=0$ the above definition is exactly the definition of Kontsevich's graph complex $\fGCc_n$.
\end{rem}

\begin{rem}
For $k=1$ we could have used the intrinsic orientation as the colored one. Than the actions of the group $S_2$ that reverse the orientation would not been needed, making the definition a bit simpler in this case.

For $k>1$ the intrinsic orientation could have been used as the first colored orientation, and the other colored orientation would be defined relative to the first one. That would make one color special.

However, we use the described definition in order to have one consistent definition for all $k$.
\end{rem}

\subsection{The differential}

The differential on $\mO_k\fGCc_n$ is defined for a graph $\Gamma$ as follows:
\begin{equation}
 \delta\Gamma\,=\sum_{t\in E}c_t(\Gamma) - \sum_{\substack{x\in V\\x\text{ 1-valent}}}d_x(\Gamma),
\end{equation}
where the map $c_t$ is ``contracting $t$'' and means putting a vertex $x$ instead of 
\begin{tikzpicture}[scale=.5]
 \node[int] (a) at (0,0) {};
 \node[int] (b) at (1,0) {};
 \draw (a) edge[->] node[above] {$\scriptstyle t$} (b);
 \node[above left] at (a) {$\scriptstyle x$};
 \node[above right] at (b) {$\scriptstyle v$};
\end{tikzpicture}
and reconnecting all edges that were previous connected to old $x$ and $v$ to the new $x$. Before contracting we permute vertices and edges in order to put vertex $v$ and the edge $t$ to be the last. If the contraction forms a cycle of any color, the resulting graph is considered to be $0$. The map $d_v$ is ``deleting $x$ and its edge'' and means deleting the vertex $x$ and one edge adjacent to it, after permuting vertices and edges in order to put vertex $x$ and the edge adjecent to it to be the last and changing the intrinsic direction of the edge towards $x$. Unless $t$ connects two 1-valent vertices, $d_x$ will cancel contracting the edge adjacent to $x$.

One can check that the differential is well defined and that $\delta^2=0$.

In the paper we will use some sub-complexes spanned by special subset of graphs, e.g.\ graphs with at least 2-valent vertices in complex $\mO_k\fGCc_n^{\geq 2}$. The definition of that complexes is straightforward, we just start from the set of particular graphs instead of $\mO_k\bar\mV_v\bar\mE_e\grac$.

\subsection{Splitting complexes and convergence of spectral sequence}

Spectral sequences will be used a lot in the paper. One wants that spectral sequence converges to the homology of the starting complex. In that case we say that spectral sequence converges \emph{correctly}.

For ensuring the correct convergence of a spectral sequence standard arguments are used, such as those from \cite[Appendix C]{DGC1}. E.g.\ we want spectral sequence to be bounded in each degree.

The differential does not change the \emph{loop number} $b:=e-v$, so the defined complexes split as the direct sum:
\begin{equation}
\mO_k\fGC_n=\bigoplus_{b\in\Z}\mB_b\mO_k\fGC_n,
\end{equation}
where $\mB_b\mO_k\fGC_n$ is the sub-complex with fixed $b=e-v$. This will be the case also for the other complexes defined later.

To show that a spectral sequence of the complex that is equal to the direct sum of simpler complexes converges correctly it is enough to show the statement for the complexes in the sum. Sub-complexes with fixed loop number mentioned above will often have bounded spectral sequences. That easily implies their correct convergence, and hence the correct convergence of the whole complex.

\subsection{Distinguishable vertices}

In the paper we will introduce differentials that does not change the number of vertices in the complex. In that case it is reasonable to consider complexes with distinguishable vertices:
\begin{equation}
\bar\mV\mO_k\fGCc_n:=\left\{
\begin{array}{ll}
\displaystyle\bigoplus_{v,e}\left(\mO_k\bar\mV_v\bar\mE_e\GS\otimes\sgn_e\right)_{\left(S_e\ltimes S_2^{e}\right)}
\qquad&\text{for $n$ even,}\\
\displaystyle\bigoplus_{v,e}\left(\mO_k\bar\mV_v\bar\mE_e\GS\otimes\sgn_2^e\right)_{\left(S_e\ltimes S_2^{e}\right)}
\qquad&\text{for $n$ odd.}
\end{array}
\right.
\end{equation}
and similarly for their sub-complexes. The original complexes are now the complexes of coinvariants of this complexes, possibly tensored with $\sgn_v$, under the action of $S_v$.

The homology in respect to the differential that does not change the number of vertices commutes with the action of the group, so to understand the homology of the original complex it is enough to understand the homology of this complex with distinguishable vertices.

\section{Sub-complexes of full oriented graph complex}
\label{s:sc}

In this section we investigate some simple sub-complexes of the full oriented graph complex $\mO_k\fGC_n$ needed in the paper. The results are generalization of \cite[Proposition 3.4]{grt} and use essentially the same ideas for proving.

\subsection{At least 2-valent vertices}

Let $\mO_k\fGCc_n^{\geq i}$ be the complex of graphs with vertices at least $i$-valent. We will particularly be interested in the case when $i=2$. One easily checks that the differential can not produce less then 2-valent vertex, so $\mO_k\fGCc_n^{\geq 2}$ is a sub-complex.

We call \emph{the passing vertex} a 2-valent vertex which is the head of one edge and the tail of another for every color.

\begin{prop}
\label{prop:2}
$$
H(\mO_k\fGCc_n)=H\left(\mO_k\fGCc_n^{\geq 2}\right)
$$
\end{prop}
\begin{proof}
The differential can neither create nor destroy 1-valent or isolated vertices. Therefore we have direct sum of complexes
$$
\mO_k\fGCc_n=\mO_k\fGCc_n^{\geq 2}\oplus \mO_k\fGCc_n^{1}
$$
where $\mO_k\fGCc_n^{1}$ is the sub-complex of graphs containing at least one 1-valent vertex, including the single vertex graph. It is enough to prove that $\mO_k\fGCc_n^{1}$ is acyclic.

We call \emph{the antenna} a maximal connected subgraph consisting of 1-valent and passing vertices in a graph.

We set up a spectral sequence on $\mO_k\fGCc_n^{1}$ on the number of edges that are not in an antenna. The spectral sequence is bounded, and hence converges correctly. The first differential is the one retracting an antenna. There is a homotopy that extends an antenna (summed over all antennas) that leads to the conclusion that the first differential is acyclic, and hence the whole differential.
\end{proof}

\subsection{No passing vertices}

Let $\mO_k\fGCc_n^{\circ}\subset\mO_k\fGCc_n^{\geq 2}$ be spanned by the graphs with only 2-valent vertices, and let $\mO_k\fGCc_n^{\varnothing}\subset\mO_k\fGCc_n^{\geq 2}$ be spanned by graphs that have at least one vertex that is at least 3-valet. Clearly, they are sub-complexes and
\begin{equation}
\mO_k\fGCc_n^{\geq 2}=\mO_k\fGCc_n^{\circ}\oplus\mO_k\fGCc_n^{\varnothing}.
\end{equation}

Let $\mO_k\fGCc_n^{\rightarrow}\subset\mO_k\fGCc_n^{\geq 2}$ be the complex of graphs with at least one passing vertex. The differential can not destroy the last passing vertex, so it is indeed the sub-complex. Let $\mO_k\fGCc_n^{\nrightarrow}$ be the quotient $\mO_k\fGCc_n^{\geq 2}/\mO_k\fGCc_n^{\rightarrow}$. For $k=0$ every 2-valent vertex is the passing vertex, so $\mO_0\fGCc_n^{\nrightarrow}\cong\mO_0\fGCc_n^{\geq 3}$. We also have sub-complexes $\mO_k\fGCc_n^{\circ\nrightarrow}$ and $\mO_k\fGCc_n^{\varnothing\nrightarrow}$. We will often use the shorter notation
\begin{equation}
\mO_k\GC_n:=\mO_k\fGCc_n^{\varnothing\nrightarrow}.
\end{equation}

\begin{prop}
\begin{enumerate}
\item[]
\item For $k\geq 0$ it holds that $H\left(\mO_k\fGCc_n^{\varnothing}\right)=H\left(\mO_k\GC_n\right)$,
\item For $k\geq 1$ it holds that $H\left(\mO_k\fGCc_n^{\circ}\right)=H\left(\mO_k\fGCc_n^{\circ\nrightarrow}\right)$.
\end{enumerate}
\end{prop}
\begin{proof}
In (1) it is enough to prove that $\mO_k\fGCc_n^{\varnothing\rightarrow}$ is acyclic.

We set up a spectral sequence on $\mO_k\fGCc_n^{\varnothing\rightarrow}$ on the number of non-passing vertices. The spectral sequence obviously converges correctly. The first differential decreases the number of passing vertices by one. There is a homotopy that extends the string of neighboring passing vertices by one, summed over all such strings, showing the acyclicity.

For (2) in $\mO_k\fGCc_n^{\circ}$ there is at least one (and therefore 2) non-passing vertices because there are no colored cycles. Therefore we can again group passing vertices into the strings of neighboring ones and do the same proof. For $k=0$ this does not work because there are no non-passing vertices in this case, and it does not make a colored cycles for there are no colors.
\end{proof}

\section{Special oriented graph complex}
\label{s:sogc}

In this section we define the special oriented graph complex $\mO_{k+1}\sGC_n$ that makes the natural codomain of the quasi-isomorphism $\mO_k\GC_n\rightarrow\mO_{k+1}\sGC_{n+1}$. It is a sub-complex (or a quotient) of $\mO_k\GC_n$, but it can also be understood as a special complex with additional kind of edges.

Let $k\geq 0$. We will now work with the complexes that have $k+1$ colors to emphasize the last color $k+1$ which will play a distinctive role. We draw it in black. In what follows, we will take the representative of the coinvariant class to have intrinsic orientation the same as orientation of the color $k+1$. This convention is important for signs.

\subsection{Skeleton graph}

Let $\Gamma\in\mO_{k+1}\fGCc_n^{\nrightarrow}$ be a graph. \emph{Weakly passing vertex} in $\Gamma$ is the vertex that is passing in all colors $c\leq k$ and not passing in the color $k+1$. Vertices in $\Gamma$ that are not weakly passing are called \emph{skeleton vertices}. Note that there is at least one skeleton vertex if $k\geq 1$ or if $\Gamma\in\mO_1\GC_n$ for $k=0$. If $k=0$ all 2-valent vertices are weakly passing, so $\Gamma\in\mO_1\fGCc_n^{\circ\nrightarrow}$ does not have skeleton vertices.

Now, for $\Gamma\in\mO_{k+1}\GC_n$ or $\Gamma\in\mO_{k+1}\fGCc_n^{\circ\nrightarrow}$ if $k\geq 1$ we consider the graph with skeleton vertices as vertices, and the string of edges and vertices between them, called \emph{skeleton edges}, as edges. We call that graph the \emph{skeleton} of $\Gamma$. Skeleton edges are the strings of edges and 2-valent weakly passing vertices, edges heading towards the same direction in all colors $c\leq k$ and alternating in the color $k+1$.

\subsection{The construction}

In skeleton, there are two kind of edges of length 2 (2 edges in original graph) regarding the color $k+1$:
\begin{tikzpicture}[baseline=-.65ex,scale=.5]
 \node[int] (a) at (0,0) {};
 \node[int] (b) at (1,0) {};
 \node[int] (c) at (2,0) {};
 \draw (a) edge[-latex] (b);
 \draw (b) edge[latex-] (c);
\end{tikzpicture} and 
\begin{tikzpicture}[baseline=-.65ex,scale=.5]
 \node[int] (a) at (0,0) {};
 \node[int] (b) at (1,0) {};
 \node[int] (c) at (2,0) {};
 \draw (a) edge[latex-] (b);
 \draw (b) edge[-latex] (c);
\end{tikzpicture}.
Regarding other colors $c\leq k$, there can be any orientation in them, but the same for both edges in original graph.
Instead of those two kind of edges, we will use the following two kinds of edges, that span the same space:
\begin{equation}
\label{def:dotted}
\begin{tikzpicture}[baseline=-.65ex,scale=.7]
 \node[int] (a) at (0,0) {};
 \node[int] (c) at (1,0) {};
 \node[above] at (a) {$\scriptstyle x$};
 \node[above] at (c) {$\scriptstyle y$};
 \draw (a) edge[dotted,->,leftred] (c);
\end{tikzpicture}
:=\frac{1}{2}\left(
\begin{tikzpicture}[baseline=-.65ex,scale=.7]
 \node[int] (a) at (0,0) {};
 \node[int] (b) at (1,0) {};
 \node[int] (c) at (2,0) {};
 \node[above] at (a) {$\scriptstyle x$};
 \node[above] at (b) {$\scriptstyle z$};
 \node[above] at (c) {$\scriptstyle y$};
 \draw (a) edge[-latex,leftred] node[below] {$\scriptstyle a-1$}(b);
 \draw (b) edge[latex-,leftred] node[below] {$\scriptstyle a$}(c);
\end{tikzpicture}
-
\begin{tikzpicture}[baseline=-.65ex,scale=.7]
 \node[int] (a) at (0,0) {};
 \node[int] (b) at (1,0) {};
 \node[int] (c) at (2,0) {};
 \node[above] at (a) {$\scriptstyle x$};
 \node[above] at (b) {$\scriptstyle z$};
 \node[above] at (c) {$\scriptstyle y$};
 \draw (a) edge[latex-,leftred] node[below] {$\scriptstyle a-1$}(b);
 \draw (b) edge[-latex,leftred] node[below] {$\scriptstyle a$}(c);
\end{tikzpicture}
\right)
\end{equation}
\begin{equation}
\begin{tikzpicture}[baseline=-.65ex,scale=.7]
 \node[int] (a) at (0,0) {};
 \node[int] (c) at (1,0) {};
 \node[above] at (a) {$\scriptstyle x$};
 \node[above] at (c) {$\scriptstyle y$};
 \draw (a) edge[leftred] (c);
 \draw (.5,.15) edge (.5,-.15);
\end{tikzpicture}
:=\frac{1}{2}\left(
\begin{tikzpicture}[baseline=-.65ex,scale=.7]
 \node[int] (a) at (0,0) {};
 \node[int] (b) at (1,0) {};
 \node[int] (c) at (2,0) {};
 \node[above] at (a) {$\scriptstyle x$};
 \node[above] at (b) {$\scriptstyle z$};
 \node[above] at (c) {$\scriptstyle y$};
 \draw (a) edge[-latex,leftred] node[below] {$\scriptstyle a-1$}(b);
 \draw (b) edge[latex-,leftred] node[below] {$\scriptstyle a$}(c);
\end{tikzpicture}
+
\begin{tikzpicture}[baseline=-.65ex,scale=.7]
 \node[int] (a) at (0,0) {};
 \node[int] (b) at (1,0) {};
 \node[int] (c) at (2,0) {};
 \node[above] at (a) {$\scriptstyle x$};
 \node[above] at (b) {$\scriptstyle z$};
 \node[above] at (c) {$\scriptstyle y$};
 \draw (a) edge[latex-,leftred] node[below] {$\scriptstyle a-1$}(b);
 \draw (b) edge[-latex,leftred] node[below] {$\scriptstyle a$}(c);
\end{tikzpicture}\right)
\end{equation}
In the picture we drew orientations in one extra color $c\leq k$ as an example. So, new edges also have colored orientation for colors $c\leq k$.

A convention of numbering original vertices and edges is necessary to be strict with signs. Recall that we have fixed intrinsic orientation to the orientation in the color $k+1$. Let all vertices in the middle of skeleton edges ($z$) come after all other vertices, all edges within them came after all other edges, and two edges in a skeleton edge come one after the other. Note that the dotted edge has a new ``intrinsic'' orientation heading from the edge $a-1$ towards the edge $a$. If edges are odd (even $n$) the change of the intrinsic orientation of dotted edges changes the sign, so dotted orientation is odd, and we have to consider it for the sign. For $n$ odd edges are even so dotted orientation is even, making it unnecessarily to draw simple arrows. Dotted edges themselves are of the same parity as vertices (even for $n$ even and odd for $n$ odd), i.e.\ switching them changes sign for $n$ odd.

It is easily seen that
\begin{equation}
\delta(\,
\begin{tikzpicture}[baseline=-.65ex,scale=.7]
 \node[int] (a) at (0,0) {};
 \node[int] (c) at (1,0) {};
 \draw (a) edge[dotted,->] (c);
\end{tikzpicture}\,
)=\,
\begin{tikzpicture}[baseline=-.65ex,scale=.7]
 \node[int] (a) at (0,0) {};
 \node[int] (c) at (1,0) {};
 \draw (a) edge[-latex] (c);
\end{tikzpicture}
\,-(-1)^n\,
\begin{tikzpicture}[baseline=-.65ex,scale=.7]
 \node[int] (a) at (0,0) {};
 \node[int] (c) at (1,0) {};
 \draw (a) edge[latex-] (c);
\end{tikzpicture},
\end{equation}
\begin{equation}
\delta(\,
\begin{tikzpicture}[baseline=-.65ex,scale=.7]
 \node[int] (a) at (0,0) {};
 \node[int] (c) at (1,0) {};
 \draw (a) edge (c);
 \draw (.5,.15) edge (.5,-.15);
\end{tikzpicture}\,
)=0.
\end{equation}
Therefore, the complexes split as
\begin{equation}
\mO_{k+1}\GC_n=\mO_{k+1}\sGC_n'\oplus\mO_{k+1}\fGCc_n^{\dagger}\qquad\text{for }k\geq 0,
\end{equation}
\begin{equation}
\mO_{k+1}\fGCc_n^{\circ\nrightarrow}=\mO_{k+1}\sGC_n^{\circ}\oplus\mO_{k+1}\fGCc_n^{\circ\dagger}\qquad\text{for }k\geq 1,
\end{equation}
where the \emph{special oriented graph complexes} $\mO_{k+1}\sGC_n'\subset\mO_{k+1}\GC_n$ and $\mO_{k+1}\sGC_n^\circ\subset\mO_{k+1}\fGCc_n^{\circ\nrightarrow}$ are the sub-complexes spanned by graphs whose skeleton edges are
$\begin{tikzpicture}[baseline=-.65ex,scale=.5]
 \node[int] (a) at (0,0) {};
 \node[int] (c) at (1,0) {};
 \draw (a) edge[-latex] (c);
\end{tikzpicture}$ or
$\begin{tikzpicture}[baseline=-.65ex,scale=.5]
 \node[int] (a) at (0,0) {};
 \node[int] (c) at (1,0) {};
 \draw (a) edge[dotted,->] (c);
\end{tikzpicture}$, and $\mO_{k+1}\fGCc_n^{\dagger}\subset\mO_{k+1}\GC_n$ and $\mO_{k+1}\fGCc_n^{\circ\dagger}\subset\mO_{k+1}\fGCc_n^{\circ\nrightarrow}$ are the sub-complexes spanned by graphs that has at least one skeleton edge of the different kind.

\begin{prop}
\label{prop:3}
\begin{enumerate}
\item[]
\item For $k\geq 0$ it holds that $H\left(\mO_{k+1}\GC_n\right)=H\left(\mO_{k+1}\sGC_n'\right)$,
\item For $k\geq 1$ it holds that $H\left(\mO_{k+1}\fGCc_n^{\circ\nrightarrow}\right)=H\left(\mO_{k+1}\sGC_n^{\circ}\right)$.
\end{enumerate}
\end{prop}
\begin{proof}
It is enough to show that $\mO_{k+1}\fGCc_n^{\dagger}$ for $k\geq 0$ and $\mO_{k+1}\fGCc_n^{\circ\dagger}$ for $k\geq 1$ are acyclic. We do that in the first case, the second one having the same proof.

On $\mO_{k+1}\fGCc_n^\dagger$ we first set a spectral sequence on the number of skeleton vertices, such that the first $\delta_0$ differential does not change the skeleton. After splitting the complex into the direct sum of complexes with fixed loop number $b=e-v$, the spectral sequences are finite in each degree, so they converge correctly, and hence the spectral sequence for the whole complex converges correctly too.

On the first page, i.e.\ on the complex $(\mO_{k+1}\fGCc_n^\dagger,\delta_0)$, the differential deletes an edge at the end of the skeleton edge (deleting in the middle would produce a passing vertex), i.e.\ it decreases a length of an edge in skeleton by 1, if it is not already 1 (deleting such an edge would change the skeleton). One can check that any skeleton edge of length 3 is mapped to $\pm
\begin{tikzpicture}[baseline=-.65ex,scale=.5]
 \node[int] (a) at (0,0) {};
 \node[int] (c) at (1,0) {};
 \draw (a) edge (c);
 \draw (.5,.15) edge (.5,-.15);
\end{tikzpicture}
$, and it is mapped to 0. There is a homotopy that does the opposite, so the first page is acyclic. Hence the result.
\end{proof}

\subsection{Alternative definition}

From now on, we will often understand $\mO_{k+1}\sGC_n'$ and its sub-complexes in an equivalent way: $\mO_{k+1}\sGC_n'$ is the graph complex with two kind of edges, \emph{solid edges}
$\begin{tikzpicture}[baseline=-.65ex,scale=.5]
 \node[int] (a) at (0,0) {};
 \node[int] (c) at (1,0) {};
 \draw (a) edge[-latex] (c);
\end{tikzpicture}$
with orientations in colors $c\leq k+1$ and \emph{dotted edges}
$\begin{tikzpicture}[baseline=-.65ex,scale=.5]
 \node[int] (a) at (0,0) {};
 \node[int] (c) at (1,0) {};
 \draw (a) edge[dotted,->] (c);
\end{tikzpicture}$ with orientations in colors $c\leq k$. Dotted edges are of the same parity as vertices (even for $n$ even and odd for $n$ odd) and have intrinsic orientation of the opposite parity. We write the color $k$ in black and skip the other colors if it does not make the confusion. The intrinsic orientation is depicted by the simple black arrow as usual. There are no cycles in any color along any kind of edges. Additionally, all vertices have to be either at least 3-valent, or not passing in one color $c\leq k$, and at least one has to be at least 3-valent. The strict definition, similar to the one from Section \ref{s:gc}, is left to the reader. The differential contracts solid edge and maps
\begin{equation}
\label{def:d0}
\delta_0:
\begin{tikzpicture}[baseline=-.65ex,scale=.7]
 \node[int] (a) at (0,0) {};
 \node[int] (c) at (1,0) {};
 \draw (a) edge[dotted,->] (c);
\end{tikzpicture}\,
\mapsto\,
\begin{tikzpicture}[baseline=-.65ex,scale=.7]
 \node[int] (a) at (0,0) {};
 \node[int] (c) at (1,0) {};
 \draw (a) edge[-latex] (c);
\end{tikzpicture}
\,-(-1)^n\,
\begin{tikzpicture}[baseline=-.65ex,scale=.7]
 \node[int] (a) at (0,0) {};
 \node[int] (c) at (1,0) {};
 \draw (a) edge[latex-] (c);
\end{tikzpicture}\,,
\end{equation}
while preserving the orientation in all other colors. There is a similar definition for $\mO_{k+1}\sGC_n^\circ$, but here all vertices have to be 2-valent.

\subsection{Tadpoles and multiple edges}
\label{ss:tadmul}
Tadpole is an edge that starts and ends at the same vertex, and multiple edge is the set of more than one edges that connect the same two vertices. In $\mO_{k+1}\sGC_n'$ we do not care of the type of an edge while talking about tadpoles and multiple edges.

In $\mO_{k+1}\sGC_n'$ there can be no tadpole of a solid edge by definition. There can not be a dotted edge if $k\geq 1$ either, because dotted edge has an orientation in colors $c\leq k$ and it would make a colored cycle.

For $n$ even there can be no tadpoles of dotted edge because of symmetry reasons. Multiple dotted edges, with possibly one solid edge, are possible, but they will not be a problem later in the paper, so we define
\begin{equation}
\label{def:DGCeven}
\mO_{k+1}\sGC_n:=\mO_{k+1}\sGC_n'\quad\text{for $n$ even.}
\end{equation}

For $n$ odd and $k=0$ a dotted tadpole is possible. It can not be destroyed by the differential, so the graphs with at least one dotted tadpole span a sub-complex. We define
\begin{equation}
\label{def:notadp}
\mO_1\sGC_n^\notadp\quad\text{for $n$ odd}
\end{equation}
to be the quotient of the complex $\mO_1\sGC_n'$ with the complex spanned by graphs with at least one tadpole.

\begin{prop}
\label{prop:notadp}
For $n$ odd the quotient
$$
\mO_1\sGC_n'\rightarrow\mO_1\sGC_n^\notadp
$$
is a quasi-isomorphism.
\end{prop}
\begin{proof}
It is enough to show that the complex spanned by graphs with at least one tadpole is acyclic.

On it, we set up a spectral sequence such that the first differential is
$$
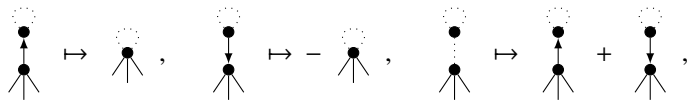
\begin{tikzpicture}[baseline=.7ex,every loop/.style={dotted}]
 \node[int] (b) at (0,.5) {};
 \node[int] (a) at (0,0) {};
 \draw (a) edge[-latex] (b);
 \draw (a) edge (0,-.4);
 \draw (a) edge (.2,-.3);
 \draw (a) edge (-.2,-.3);
 \draw (b) edge[loop] (b);
\end{tikzpicture}
\mapsto
\begin{tikzpicture}[baseline=-.7ex,every loop/.style={dotted}]
 \node[int] (a) at (0,0) {};
 \draw (a) edge (0,-.4);
 \draw (a) edge (.2,-.3);
 \draw (a) edge (-.2,-.3);
 \draw (a) edge[loop] (a);
\end{tikzpicture},
\quad
\begin{tikzpicture}[baseline=.7ex,every loop/.style={dotted}]
 \node[int] (b) at (0,.5) {};
 \node[int] (a) at (0,0) {};
 \draw (a) edge[latex-] (b);
 \draw (a) edge (0,-.4);
 \draw (a) edge (.2,-.3);
 \draw (a) edge (-.2,-.3);
 \draw (b) edge[loop] (b);
\end{tikzpicture}
\mapsto\,-
\begin{tikzpicture}[baseline=-.7ex,every loop/.style={dotted}]
 \node[int] (a) at (0,0) {};
 \draw (a) edge (0,-.4);
 \draw (a) edge (.2,-.3);
 \draw (a) edge (-.2,-.3);
 \draw (a) edge[loop] (a);
\end{tikzpicture},
\quad
\begin{tikzpicture}[baseline=.7ex,every loop/.style={dotted}]
 \node[int] (b) at (0,.5) {};
 \node[int] (a) at (0,0) {};
 \draw (a) edge[dotted] (b);
 \draw (a) edge (0,-.4);
 \draw (a) edge (.2,-.3);
 \draw (a) edge (-.2,-.3);
 \draw (b) edge[loop] (b);
\end{tikzpicture}
\mapsto
\begin{tikzpicture}[baseline=.7ex,every loop/.style={dotted}]
 \node[int] (b) at (0,.5) {};
 \node[int] (a) at (0,0) {};
 \draw (a) edge[-latex] (b);
 \draw (a) edge (0,-.4);
 \draw (a) edge (.2,-.3);
 \draw (a) edge (-.2,-.3);
 \draw (b) edge[loop] (b);
\end{tikzpicture}
+
\begin{tikzpicture}[baseline=.7ex,every loop/.style={dotted}]
 \node[int] (b) at (0,.5) {};
 \node[int] (a) at (0,0) {};
 \draw (a) edge[latex-] (b);
 \draw (a) edge (0,-.4);
 \draw (a) edge (.2,-.3);
 \draw (a) edge (-.2,-.3);
 \draw (b) edge[loop] (b);
\end{tikzpicture},
$$
summed over all tadpoles. The precise set up of the filtration and the argument that the spectral sequence converges correctly is left to the reader.

There is a homotopy mapping
$$
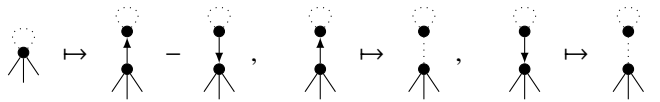
\begin{tikzpicture}[baseline=-.7ex,every loop/.style={dotted}]
 \node[int] (a) at (0,0) {};
 \draw (a) edge (0,-.4);
 \draw (a) edge (.2,-.3);
 \draw (a) edge (-.2,-.3);
 \draw (a) edge[loop] (a);
\end{tikzpicture}
\mapsto
\begin{tikzpicture}[baseline=.7ex,every loop/.style={dotted}]
 \node[int] (b) at (0,.5) {};
 \node[int] (a) at (0,0) {};
 \draw (a) edge[-latex] (b);
 \draw (a) edge (0,-.4);
 \draw (a) edge (.2,-.3);
 \draw (a) edge (-.2,-.3);
 \draw (b) edge[loop] (b);
\end{tikzpicture}
-
\begin{tikzpicture}[baseline=.7ex,every loop/.style={dotted}]
 \node[int] (b) at (0,.5) {};
 \node[int] (a) at (0,0) {};
 \draw (a) edge[latex-] (b);
 \draw (a) edge (0,-.4);
 \draw (a) edge (.2,-.3);
 \draw (a) edge (-.2,-.3);
 \draw (b) edge[loop] (b);
\end{tikzpicture},
\quad
\begin{tikzpicture}[baseline=.7ex,every loop/.style={dotted}]
 \node[int] (b) at (0,.5) {};
 \node[int] (a) at (0,0) {};
 \draw (a) edge[-latex] (b);
 \draw (a) edge (0,-.4);
 \draw (a) edge (.2,-.3);
 \draw (a) edge (-.2,-.3);
 \draw (b) edge[loop] (b);
\end{tikzpicture}
\mapsto
\begin{tikzpicture}[baseline=.7ex,every loop/.style={dotted}]
 \node[int] (b) at (0,.5) {};
 \node[int] (a) at (0,0) {};
 \draw (a) edge[dotted] (b);
 \draw (a) edge (0,-.4);
 \draw (a) edge (.2,-.3);
 \draw (a) edge (-.2,-.3);
 \draw (b) edge[loop] (b);
\end{tikzpicture},
\quad
\begin{tikzpicture}[baseline=.7ex,every loop/.style={dotted}]
 \node[int] (b) at (0,.5) {};
 \node[int] (a) at (0,0) {};
 \draw (a) edge[latex-] (b);
 \draw (a) edge (0,-.4);
 \draw (a) edge (.2,-.3);
 \draw (a) edge (-.2,-.3);
 \draw (b) edge[loop] (b);
\end{tikzpicture}
\mapsto
\begin{tikzpicture}[baseline=.7ex,every loop/.style={dotted}]
 \node[int] (b) at (0,.5) {};
 \node[int] (a) at (0,0) {};
 \draw (a) edge[dotted] (b);
 \draw (a) edge (0,-.4);
 \draw (a) edge (.2,-.3);
 \draw (a) edge (-.2,-.3);
 \draw (b) edge[loop] (b);
\end{tikzpicture},
$$
summed over all tadpoles, showing that the first differential makes the complex acyclic. That concludes the proof.
\end{proof}

Let $\mO_{k+1}\sGC_n^\notadp:=\mO_{k+1}\sGC_n'$ for $k\geq 1$. In $\mO_{k+1}\sGC_n^\notadp$ there are multiple solid edges, possibly with one dotted edge. The differential can not destroy a multiple edge, so graphs with at least one of them span the sub-complex.

\begin{rem}
For odd $n$ the double edge
\begin{tikzpicture}[baseline=-.65ex,scale=.7]
 \node[int] (a) at (0,0) {};
 \node[int] (b) at (1,0) {};
 \draw (a) edge[bend left=20,dotted] (b);
 \draw (a) edge[bend right=20,-latex] (b);
\end{tikzpicture}
can be destroyed in $\mO_1\sGC_n'$ by making a dotted tadpole. This is the very reason why do we switch to $\mO_1\sGC_n^\notadp$ first.
\end{rem}

Now we define
\begin{equation}
\label{def:sGCodd}
\mO_{k+1}\sGC_n\quad\text{for $n$ odd}
\end{equation}
to be the quotient of the complex $\mO_{k+1}\sGC_n^\notadp$ with the complex spanned by graphs with at least one multiple edge.

\begin{prop}
\label{prop:nomul}
For $k\geq 0$ and $n$ odd the quotient
$$
\mO_{k+1}\sGC_n^\notadp\rightarrow\mO_{k+1}\sGC_n
$$
is a quasi-isomorphism.
\end{prop}
\begin{proof}
It is enough to show that the complex spanned by graphs with at least one multiple edge is acyclic. On it we set up a spectral sequence on the number of skeleton vertices, such that $\delta_0$ is the first differential.
After splitting the complex into the direct product of complexes with fixed loop number, the spectral sequences are finite in each degree, so they converge correctly, and hence the spectral sequence for the whole complex too.

Since $\delta_0$ does not change the number of vertices, the homology commutes with permuting vertices and to show that the complex is acyclic it is enough to show that the complex with distinguishable vertices is acyclic. So from now on we distinguish them.

On the first page we set up another spectral sequences, on the number of dotted edges that are not in the multiple edge. This spectral sequence is finite and converges correctly. The first differential of this spectral sequence is the one that acts only on an multiple edge.

Symmetry reasons reduce possibilities of multiple edges to
$$
\begin{tikzpicture}[baseline=-.65ex]
 \node[int] (a) at (0,0) {};
 \node[int] (b) at (1,0) {};
 \draw (a) edge[very thick,-latex] node[above] {$\scriptstyle m$} (b);
\end{tikzpicture}
\,,\quad
\begin{tikzpicture}[baseline=-.65ex]
 \node[int] (a) at (0,0) {};
 \node[int] (b) at (1,0) {};
 \draw (a) edge[very thick,latex-] node[above] {$\scriptstyle m$} (b);
\end{tikzpicture}
\,,\quad
\begin{tikzpicture}[baseline=-.65ex]
 \node[int] (a) at (0,0) {};
 \node[int] (b) at (1,0) {};
 \draw (a) edge[very thick,-latex, bend left=20] node[above] {$\scriptstyle m-1$} (b);
 \draw (a) edge[dotted, bend right=20] (b);
\end{tikzpicture}
\,,\quad\text{and}\quad
\begin{tikzpicture}[baseline=-.65ex]
 \node[int] (a) at (0,0) {};
 \node[int] (b) at (1,0) {};
 \draw (a) edge[very thick,latex-, bend left=20] node[above] {$\scriptstyle m-1$} (b);
 \draw (a) edge[dotted, bend right=20] (b);
\end{tikzpicture}
$$
for $m\geq 2$ where the thick arrowed line with a number $l$ represent $l$ solid edges with the orientation of the arrow. Additionally, edges in the multiple edge have an orientation in every color $c\leq k$. All edges have it the same because of the no-loop condition, so we can talk about the orientation of a multiple edge.

The differential maps:
$$
\begin{tikzpicture}[baseline=-.65ex]
 \node[int] (a) at (0,0) {};
 \node[int] (b) at (1,0) {};
 \draw (a) edge[very thick,-latex, bend left=20] node[above] {$\scriptstyle m-1$} (b);
 \draw (a) edge[dotted, bend right=20] (b);
\end{tikzpicture}
\mapsto
\begin{tikzpicture}[baseline=-.65ex]
 \node[int] (a) at (0,0) {};
 \node[int] (b) at (1,0) {};
 \draw (a) edge[very thick,-latex] node[above] {$\scriptstyle m$} (b);
\end{tikzpicture}\,,
\quad
\begin{tikzpicture}[baseline=-.65ex]
 \node[int] (a) at (0,0) {};
 \node[int] (b) at (1,0) {};
 \draw (a) edge[very thick,latex-, bend left=20] node[above] {$\scriptstyle m-1$} (b);
 \draw (a) edge[dotted, bend right=20] (b);
\end{tikzpicture}
\mapsto
\begin{tikzpicture}[baseline=-.65ex]
 \node[int] (a) at (0,0) {};
 \node[int] (b) at (1,0) {};
 \draw (a) edge[very thick,latex-] node[above] {$\scriptstyle m$} (b);
\end{tikzpicture}\,,
$$
leaving the orientation of colors $c\leq k$ the same.
The complex with this differential is clearly acyclic, concluding the proof.
\end{proof}

\section{The construction of the quasi-isomorphism}
\label{s:qi}

In this section we prove Theorem \ref{thm:main}, that there is a quasi-isomorphism $\mO_k\GC_n\rightarrow\mO_{k+1}\GC_{n+1}$ for $k\geq 0$.

\subsection{The map}
\label{ss:map}

Let $k\geq 0$ and $\Gamma\in\mO_k\GC_n$ be a graph. We call a \emph{spanning tree} of $\Gamma$ a connected sub-graph without cycles which contains all its vertices. Let $S(\Gamma)$ be the set of all spanning trees of $\Gamma$.

For a chosen vertex $x\in V(\Gamma)$ and spanning tree $\tau\in S(\Gamma)$ we define $h_{x,\tau}(\Gamma)\in\mO_{k+1}\GC_{n+1}$ as follows.
The skeleton of $h_{x,\tau}(\Gamma)$ is isomorphic to $\Gamma$, edges which are in $\tau$ are mapped to
$\begin{tikzpicture}[baseline=-.65ex,scale=.5]
 \node[int] (a) at (0,0) {};
 \node[int] (c) at (1,0) {};
 \draw (a) edge[-latex] (c);
\end{tikzpicture}$
with the orientation away from the vertex $x$ in the color $k+1$ and edges which are not in $\tau$ are mapped to
$\begin{tikzpicture}[baseline=-.65ex,scale=.5]
 \node[int] (a) at (0,0) {};
 \node[int] (c) at (1,0) {};
 \draw (a) edge[->,dotted] (c);
\end{tikzpicture}$, \eqref{def:dotted}, with preserved intrinsic orientation. Orientations in other colors $c\leq k$ are preserved.

In order to precisely define the sign, we make a convention that vertices of $h_{x,\tau}(\Gamma)$ take labels from edges in $\Gamma$ and edges of $h_{x,\tau}(\Gamma)$ take labels from vertices in $\Gamma$ in the following sense.

Recall that vertices are labeled with numbers starting from $0$ and edges are labeled with numbers starting from $1$.
First we permute vertices of $\Gamma$ such that the chosen vertex $x$ is the vertex $0$. That vertex in $h_{x,\tau}(\Gamma)$ is labeled also $0$. Other vertices in the skeleton of $h_{x,\tau}(\Gamma)$ are labeled by the former label (in $\Gamma$) of the edge heading towards it. Vertices in the middle of a dotted skeleton edge are labeled by the former label of that edge. So, all vertices take label $0$, and labels of all edges from $\Gamma$.

Solid edges in the skeleton of $h_{x,\tau}(\Gamma)$ are labeled with the former label of the vertex to which they head. The following numbers are used to label remaining edges: two edges in dotted skeleton edges are labeled with two consecutive numbers, such that lower number is taken by the edge at the intrinsic tail of the former edge (in $\Gamma$). The order of the pairs of edges in each dotted edge does not matter in any parity case.

Everything gets a pre-factor $(-1)^{nr}$ where $r$ is the number of edges in $\tau$ whose orientation in the color $k+1$ is the opposite to the former intrinsic orientation.

The following diagram describes an example of $h_{x,\tau}$.
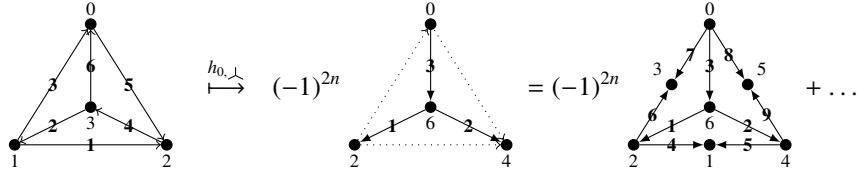
$$
\begin{tikzpicture}[baseline=1ex]
 \node[int] (a) at (-1,-.5) {};
 \node[int] (b) at (1,-.5) {};
 \node[int] (c) at (0,0) {};
 \node[int] (d) at (0,1.1) {};
 \node[below] at (a) {$\scriptstyle 1$};
 \node[below] at (b) {$\scriptstyle 2$};
 \node[below] at (c) {$\scriptstyle 3$};
 \node[above] at (d) {$\scriptstyle 0$};
 \draw (a) edge[->] node {$\scriptstyle {\mathbf 1}$} (b);
 \draw (a) edge[<-] node {$\scriptstyle {\mathbf 2}$} (c);
 \draw (a) edge[->] node {$\scriptstyle {\mathbf 3}$} (d);
 \draw (b) edge[->] node {$\scriptstyle {\mathbf 4}$} (c);
 \draw (b) edge[<-] node {$\scriptstyle {\mathbf 5}$} (d);
 \draw (c) edge[->] node {$\scriptstyle {\mathbf 6}$} (d);
\end{tikzpicture}\quad
\mxto{h_{0,
\begin{tikzpicture}[baseline=0ex,scale=.1]
 \draw (-1,-.5) edge (0,0);
 \draw (1,-.5) edge (0,0);
 \draw (0,0) edge (0,1.1);
\end{tikzpicture}}}
\quad (-1)^{2n}
\begin{tikzpicture}[baseline=1ex]
 \node[int] (a) at (-1,-.5) {};
 \node[int] (b) at (1,-.5) {};
 \node[int] (c) at (0,0) {};
 \node[int] (d) at (0,1.1) {};
 \node[below] at (a) {$\scriptstyle 2$};
 \node[below] at (b) {$\scriptstyle 4$};
 \node[below] at (c) {$\scriptstyle 6$};
 \node[above] at (d) {$\scriptstyle 0$};
 \draw (a) edge[latex-] node {$\scriptstyle {\mathbf 1}$} (c);
 \draw (b) edge[latex-] node {$\scriptstyle {\mathbf 2}$} (c);
 \draw (c) edge[latex-] node {$\scriptstyle {\mathbf 3}$} (d);
 \draw (a) edge[->,dotted] (b);
 \draw (a) edge[->,dotted] (d);
 \draw (b) edge[<-,dotted] (d);
\end{tikzpicture}
=(-1)^{2n}
\begin{tikzpicture}[baseline=1ex]
 \node[int] (a) at (-1,-.5) {};
 \node[int] (b) at (1,-.5) {};
 \node[int] (c) at (0,0) {};
 \node[int] (d) at (0,1.1) {};
 \node[below] at (a) {$\scriptstyle 2$};
 \node[below] at (b) {$\scriptstyle 4$};
 \node[below] at (c) {$\scriptstyle 6$};
 \node[above] at (d) {$\scriptstyle 0$};
 \draw (a) edge[latex-] node {$\scriptstyle {\mathbf 1}$} (c);
 \draw (b) edge[latex-] node {$\scriptstyle {\mathbf 2}$} (c);
 \draw (c) edge[latex-] node {$\scriptstyle {\mathbf 3}$} (d);
 \node[int] (e) at (0,-.5) {};
 \node[int] (f) at (-.5,.3) {};
 \node[int] (g) at (.5,.3) {};
 \node[below] at (e) {$\scriptstyle 1$};
 \node[above left] at (f) {$\scriptstyle 3$};
 \node[above right] at (g) {$\scriptstyle 5$};
 \draw (a) edge[-latex] node {$\scriptstyle {\mathbf 4}$} (e);
 \draw (b) edge[-latex] node {$\scriptstyle {\mathbf 5}$} (e);
 \draw (a) edge[-latex] node {$\scriptstyle {\mathbf 6}$} (f);
 \draw (d) edge[-latex] node {$\scriptstyle {\mathbf 7}$} (f);
 \draw (b) edge[-latex] node {$\scriptstyle {\mathbf 9}$} (g);
 \draw (d) edge[-latex] node {$\scriptstyle {\mathbf 8}$} (g); 
\end{tikzpicture}+\dots
$$

One can check that the map is well defined (sign change from a permutation or reversing in $\Gamma$ is the same after mapping), cycles are never formed in any color and that the degree of the map is $0$.

Let
\begin{equation}
\label{def:h}
h(\Gamma):=\sum_{x\in V(\Gamma)}(v(x)-2)\sum_{\tau\in S(\Gamma)}h_{x,\tau}(\Gamma),
\end{equation}
where $v(x)$ is the valence of the vertex $x$ in $\Gamma$.

\begin{prop}
The map $h\colon \mO_k\GC_n\rightarrow\mO_{k+1}\GC_{n+1}$ is a map of complexes, i.e.\ $\delta h(\Gamma)=h(\delta\Gamma)$ for every $\Gamma\in\mO_k\GC_n$.
\end{prop}
\begin{proof}
It holds that
$$
 h(\delta\Gamma)=h\left(\sum_{t\in E(\Gamma)}c_t(\Gamma)\right)=\sum_{t\in E(\Gamma)}h\left(c_t(\Gamma)\right)=
 \sum_{t\in E(\Gamma)}\sum_{x\in V(c_t(\Gamma))}\left(v_{c_t(\Gamma)}(x)-2\right)\sum_{\tau\in S(c_t(\Gamma))}h_{x,\tau}(c_t(\Gamma))
$$
where $c_t(\Gamma)$ is contracting an edge $t$ in $\Gamma$. Spanning trees of $c_t(\Gamma)$ are in natural bijection with spanning trees of $\Gamma$ which contain $t$, $c_t(\tau)\leftrightarrow\tau$, so we can write
$$
 h(\delta\Gamma)=\sum_{t\in E(\Gamma)}\sum_{\substack{\tau\in S(\Gamma)\\t\in E(\tau)}}\sum_{x\in V(c_t(\Gamma))}
 \left(v_{c_t(\Gamma)}(x)-2\right)h_{x,c_t(\tau)}(c_t(\Gamma))=
 \sum_{\tau\in S(\Gamma)}\sum_{t\in E(\tau)}\sum_{x\in V(c_t(\Gamma))}
\left(v_{c_t(\Gamma)}(x)-2\right)h_{x,c_t(\tau)}(c_t(\Gamma)).
$$
For $x\in V(c_t(\Gamma))$ we have two choices, either $x\in V(\Gamma)$ not adjacent to $t$, or $x$ is produced by contraction of $t$. It can be seen (after being careful with signs) that in the first case $h_{x,c_t(\tau)}(c_t(\Gamma))=c_t(h_{x,\tau}(\Gamma))$ and in the second case $h_{x,c_t(\tau)}(c_t(\Gamma))=c_t\left(h_{\Gamma_-(t),\tau}(\Gamma)\right)=c_t\left(h_{\Gamma_+(t),\tau}(\Gamma)\right)$, where $\Gamma_-(t)$ and $\Gamma_+(t)$ are ends of the edge $t$.
Since in the first case vertices do not change valence after applying $c_t$, and it is $v_{c_t(\Gamma)}(x)=v_\Gamma(\Gamma_-(t))+v_\Gamma(\Gamma_+(t))-2$ in the second case, it holds that
\begin{multline*}
 h(\delta\Gamma)=\sum_{\tau\in S(\Gamma)}\sum_{t\in E(\tau)}
 \left(\sum_{\substack{x\in V(\Gamma)\\x\notin\{\Gamma_-(t),\Gamma_+(t)\}}}(v_{\Gamma}(x)-2)c_t(h_{x,\tau}(\Gamma))+(v_\Gamma(\Gamma_-(t))+v_\Gamma(\Gamma_+(t))-4)c_t(h_{\Gamma_-(t),\tau}(\Gamma))\right)=\\
 =\sum_{\tau\in S(\Gamma)}\sum_{t\in E(\tau)}\sum_{x\in V(\Gamma)}(v_{\Gamma}(x)-2)c_t(h_{x,\tau}(\Gamma))=
 \sum_{x\in V(\Gamma)}(v_{\Gamma}(x)-2)\sum_{\tau\in S(\Gamma)}\sum_{t\in E(\tau)}c_t(h_{x,\tau}(\Gamma)) .
\end{multline*}

On the other side
$$
\delta h(\Gamma)=\delta\left(\sum_{x\in V(\Gamma)}(v_\Gamma(x)-2)\sum_{\tau\in S(\Gamma)}h_{x,\tau}(\Gamma)\right)=
\sum_{x\in V(\Gamma)}(v_\Gamma(x)-2)\sum_{\tau\in S(\Gamma)}\sum_{t\in E(h_{x,\tau}(\Gamma))}c_t\left(h_{x,\tau}(\Gamma)\right).
$$
The edge $t\in E(h_{x,\tau}(\Gamma))$ can be chosen in two ways, it is either in the spanning tree $\tau$ or one of the edges of doted skeleton edge. In the first case the sum gives exactly $h(\delta\Gamma)$. Note that if contracting a skeleton edge makes a cycle in any color $c\leq k$, it makes so also after the mapping. Therefore
$$
\delta h(\Gamma)=h(\delta\Gamma)+\sum_{x\in V(\Gamma)}(v_\Gamma(x)-2)\sum_{\tau\in S(\Gamma)}\sum_{T\in D(\Gamma,\tau)}C_T\left(h_{x,\tau}(\Gamma)\right)
$$
where $D(\Gamma,\tau):=E(\Gamma)\setminus E(\tau)$ is bijective with the set of dotted edges in skeleton of $h_{x,\tau}(\Gamma)$ and $C_T\left(h_{x,\tau}(\Gamma)\right)$ is a graph in $\mO_{k+1}\sGC_{n+1}$ given from $h_{x,\tau}(\Gamma)$ by replacing dotted edge which come from $T$ with
$\begin{tikzpicture}[baseline=-.65ex,scale=.5]
 \node[int] (a) at (0,0) {};
 \node[int] (c) at (1,0) {};
 \draw (a) edge[-latex] (c);
\end{tikzpicture}
-(-1)^{n+1}
\begin{tikzpicture}[baseline=-.65ex,scale=.5]
 \node[int] (a) at (0,0) {};
 \node[int] (c) at (1,0) {};
 \draw (a) edge[latex-] (c);
\end{tikzpicture}$,
and is $0$ if that produces a cycle in color $k+1$. It cannot produce a cycle in any other color $c\leq k$.

To finish the proof it is enough to show that
\begin{equation}
N(\Gamma,x):=\sum_{\tau\in S(\Gamma)}\sum_{T\in D(\Gamma,\tau)}C_T\left(h_{x,\tau}(\Gamma)\right)
\end{equation}
is zero for every $x\in V(\Gamma)$. Terms in above relation can be summed in another order. Let $CT(\Gamma)$ be the set of all connected sub-graphs $\sigma$ of $\Gamma$ which contain all vertices and which has one more edge than spanning tree (has one cycle), and let $C(\sigma)$ be the set of edges in the cycle of $\sigma$. Clearly, $\sigma\setminus T$ for $T\in C(\sigma)$ is a spanning tree of $\Gamma$ and sets $\{(\tau,T)|\tau\in S(\Gamma),T\in D(\Gamma,\tau)\}$ and $\{(\sigma,T)|\sigma\in CT(\Gamma),T\in C(\sigma)\}$ are bijective, so
\begin{equation}
N(\Gamma,x)=\sum_{\sigma\in CT(\Gamma)}\sum_{T\in C(\sigma)}C_T\left(h_{x,\sigma\setminus T}(\Gamma)\right).
\end{equation}
It is now enough to show that $\sum_{T\in C(\sigma)}C_T\left(h_{x,\sigma\setminus T}(\Gamma)\right)=0$ for every $x\in V(\Gamma)$ and for every $\sigma\in CT(\Gamma)$.
Let $y\in V(\Gamma)$ be the vertex in cycle of $\sigma$ closest to vertex $x$ (along $\sigma$). After choosing $T\in C(\sigma)$, that cycle in $h_{x,\sigma\setminus T}(\Gamma)$ has one dotted edge, and directions of other edges go from $y$ to that dotted edge, such as on the diagram.
$$
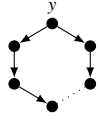
\begin{tikzpicture}[baseline=0ex,scale=.5]
 \node[int] (a) at (0,1.1) {};
 \node[int] (b) at (1,.5) {};
 \node[int] (c) at (1,-.5) {};
 \node[int] (d) at (0,-1.1) {};
 \node[int] (e) at (-1,-.5) {};
 \node[int] (f) at (-1,.5) {};
 \node[above] at (a) {$\scriptstyle y$};
 \draw (a) edge[-latex] (b);
 \draw (b) edge[-latex] (c);
 \draw (a) edge[-latex] (f);
 \draw (f) edge[-latex] (e);
 \draw (e) edge[-latex] (d);
 \draw (c) edge[dotted] (d);
\end{tikzpicture}
$$
After acting by $C_T$ the dotted edge is replaced by
$\begin{tikzpicture}[baseline=-.65ex,scale=.5]
 \node[int] (a) at (0,0) {};
 \node[int] (c) at (1,0) {};
 \draw (a) edge[-latex] (c);
\end{tikzpicture}
+(-1)^n
\begin{tikzpicture}[baseline=-.65ex,scale=.5]
 \node[int] (a) at (0,0) {};
 \node[int] (c) at (1,0) {};
 \draw (a) edge[latex-] (c);
\end{tikzpicture}$,
like in the diagram
$$
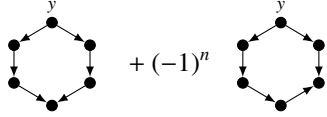
\begin{tikzpicture}[baseline=-.6ex,scale=.5]
 \node[int] (a) at (0,1.1) {};
 \node[int] (b) at (1,.5) {};
 \node[int] (c) at (1,-.5) {};
 \node[int] (d) at (0,-1.1) {};
 \node[int] (e) at (-1,-.5) {};
 \node[int] (f) at (-1,.5) {};
 \node[above] at (a) {$\scriptstyle y$};
 \draw (a) edge[-latex] (b);
 \draw (b) edge[-latex] (c);
 \draw (a) edge[-latex] (f);
 \draw (f) edge[-latex] (e);
 \draw (e) edge[-latex] (d);
 \draw (c) edge[-latex] (d);
\end{tikzpicture}
\quad+(-1)^n\quad
\begin{tikzpicture}[baseline=-.6ex,scale=.5]
 \node[int] (a) at (0,1.1) {};
 \node[int] (b) at (1,.5) {};
 \node[int] (c) at (1,-.5) {};
 \node[int] (d) at (0,-1.1) {};
 \node[int] (e) at (-1,-.5) {};
 \node[int] (f) at (-1,.5) {};
 \node[above] at (a) {$\scriptstyle y$};
 \draw (a) edge[-latex] (b);
 \draw (b) edge[-latex] (c);
 \draw (a) edge[-latex] (f);
 \draw (f) edge[-latex] (e);
 \draw (e) edge[-latex] (d);
 \draw (c) edge[latex-] (d);
\end{tikzpicture}
$$
Careful calculation of the sign shows that those two terms are canceled with terms given from choosing neighboring dotted edges, and two last terms which does not have corresponding neighbor are indeed $0$ because they have cycle in color $k+1$. Therefore
\begin{equation}
\sum_{T\in C(\sigma)}C_T\left(h_{x,\sigma\setminus T}(\Gamma)\right)=0 
\end{equation}
for every $x\in V(\Gamma)$ and for every $\sigma\in CT(\Gamma)$, so $N(\Gamma,x)=0$.
\end{proof}

Clearly, $h(\Gamma)$ always sits in $\mO_{k+1}\sGC_{n+1}'$, so we can define a co-restriction $h:\mO_k\GC_n\rightarrow\mO_{k+1}\sGC_{n+1}'$ and for $n$ even the map to the quotient $h:\mO_k\GC_n\rightarrow\mO_{k+1}\sGC_{n+1}$, recall \eqref{def:notadp} and \eqref{def:sGCodd}. By the abuse of notation, we call all these maps $h$. Note that for $n$ even ($n+1$ odd) no graph is sent to $0$-class (graph with tadpole or multiple edge) because in that case there are no multiple edges in $\mO_k\GC_n$.

\subsection{The proof}

\begin{prop}
The map $h:\mO_k\GC_n\rightarrow\mO_{k+1}\sGC_{n+1}$ is a quasi-isomorphism.
\end{prop}
\begin{proof}
On the mapping cone of $h$ we set up the standard spectral sequence, on the number of vertices (in $\mO_{k+1}\sGC_{n+1}$ skeleton vertices). Standard splitting of complexes as the product of complexes with fixed loop number implies the correct convergence.

On the first page of the spectral sequence there is a mapping cone of the map
$$
h:\left(\mO_k\GC_n,0\right)\rightarrow\left(\mO_{k+1}\sGC_{n+1},\delta_0\right).
$$
Since no differential changes the number of vertices, the homology commutes with permuting vertices and to show that this mapping cone is acyclic, it is enough to show that for distinguishable vertices, i.e.\ for the map
$$
h:\left(\bar\mV\mO_k\GC_n,0\right)\rightarrow\left(\bar\mV\mO_{k+1}\sGC_{n+1},\delta_0\right).
$$

Edges are still not distinguishable. But we can talk about the \emph{positional edge}, that is a pair of vertices in $V(\Gamma)\times V(\Gamma)$ for which there is at least one edge between them.

For $n$ even ($n+1$ odd) there are no multiple edges in $\bar\mV\mO_k\GC_n$ by symmetry nor in $\bar\mV\mO_{k+1}\sGC_{n+1}$ by definition, so positional edges are indeed edges, and they can be distinguished.

For $n$ odd ($n+1$ even) there can be multiple edges both in $\bar\mV\mO_k\GC_n$ and $\bar\mV\mO_{k+1}\sGC_{n+1}$, so one positional edge can consist of more edges and they are undistinguishable if of the same type. In $\bar\mV\mO_{k+1}\sGC_{n+1}$ symmetry reasons allow only the following positional edges:
$$
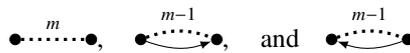
\begin{tikzpicture}[baseline=-.65ex]
 \node[int] (a) at (0,0) {};
 \node[int] (b) at (1,0) {};
 \draw (a) edge[very thick,dotted] node[above] {$\scriptstyle m$} (b);
\end{tikzpicture}
,\quad
\begin{tikzpicture}[baseline=-.65ex]
 \node[int] (a) at (0,0) {};
 \node[int] (b) at (1,0) {};
 \draw (a) edge[very thick,dotted, bend left=20] node[above] {$\scriptstyle m-1$} (b);
 \draw (a) edge[-latex, bend right=20] (b);
\end{tikzpicture}
,\quad\text{and}\quad
\begin{tikzpicture}[baseline=-.65ex]
 \node[int] (a) at (0,0) {};
 \node[int] (b) at (1,0) {};
 \draw (a) edge[very thick,dotted, bend left=20] node[above] {$\scriptstyle m-1$} (b);
 \draw (a) edge[latex-, bend right=20] (b);
\end{tikzpicture}
$$
for $m\geq 1$, where the thick dotted line with a number $l$ represent $l$ dotted lines. Additionally, a positional edge has an orientation in colors $c\leq k$ both in $\bar\mV\mO_k\GC_n$ and $\bar\mV\mO_{k+1}\sGC_{n+1}$.

With distinguishable vertices we can also chose a preferred orientation of an edge (e.g.\ from the vertex with a smaller number towards the vertex with the bigger number).

Let the \emph{shape}
\begin{equation}
s:\bar\mV\mO_{k+1}\sGC_{n+1}\rightarrow\bar\mV\mO_k\GC_n
\end{equation}
map a graph $\Gamma$ to the graph with the same vertices as the skeleton, and with the same edges, but forgetting the type of the edge and the orientation in the color $k+1$, while still remembering the orientation in other colors.

The differential $\delta_0$ does not change the shape, just the type of an edge. Therefore the mapping cone of $h:\left(\bar\mV\mO_k\GC_n,0\right)\rightarrow\left(\bar\mV\mO_{k+1}\sGC_{n+1},\delta_0\right)$ splits as a direct product of mapping cones of maps for each shape $\Sigma\in\bar\mV\mO_k\GC_n$, that is of

\begin{equation}
\label{core}
h:\left(\K\Sigma,0\right)\rightarrow\left(\bar\mV\mO_{k+1}\sGC_{n+1}^\Sigma,\delta_0\right)
\end{equation}
where $\bar\mV\mO_{k+1}\sGC_{n+1}^\Sigma$ is the complex spanned by all graphs $\Gamma\in\bar\mV\mO_{k+1}\sGC_{n+1}$ such that $s(\Gamma)=\Sigma$.

It is now enough to show that \eqref{core} is a quasi-isomorphism. Homology of the left-hand complex is clearly one-dimensional, so we need to show that the homology of the right-hand side is also one dimensional, and that $\Sigma$ is mapped to a representative of the class.

Let us now fix a shape $\Sigma$ with $v$ vertices. In $\Sigma$ we choose one vertex $y$ and $v-1$ positional edges, say $e_1,\dots,e_{v-1}$, such that for every $i=1,\dots,v-1$ the edges $\{e_1,\dots, e_i\}$ form a sub-tree of $\Sigma$ that contains vertex $y$.

Let $\Sigma^i$ for $i=0,\dots,v-1$ be the complex defined as follows. It is spanned by graphs with shape $\Sigma$ and three types of edges. There is one thick edge
\begin{tikzpicture}[baseline=-.65ex,scale=.5]
 \node[int] (a) at (0,0) {};
 \node[int] (c) at (1,0) {};
 \draw (a) edge[very thick] (c);
\end{tikzpicture}
on the position of each positional edge $\{e_1,\dots e_i\}$ with the orientations in colors $c\leq k$, and solid and dotted edges on the other skeleton edges, solid being oriented in colors $c\leq k+1$ and dotted being colored in colors $c\leq k$, such that there are no cycles in any color $c\leq k+1$ where thick edges are considered to have both orientations in the color $k+1$. The differential $\delta_0$ acts on dotted edges as usual, while the thick edges stay always unchanged. 
The strict definition is left to the reader.
An example for $k=0$ is shown in Figure \ref{fig:complexSigma}.

\begin{figure}[h]
$$
\begin{tikzpicture}[baseline=1ex]
 \node[int] (a) at (-1,-.5) {};
 \node[int] (b) at (1,-.5) {};
 \node[int] (c) at (0,0) {};
 \node[int] (d) at (0,1.1) {};
 \draw (a) edge (b);
 \draw (a) edge (c);
 \draw (a) edge (d);
 \draw (b) edge (c);
 \draw (b) edge (d);
 \draw (c) edge[bend left=15] (d);
 \draw (c) edge[bend right=15] (d);
\end{tikzpicture}\quad\quad
\begin{tikzpicture}[baseline=1ex]
 \node[int] (a) at (-1,-.5) {};
 \node[int] (b) at (1,-.5) {};
 \node[int] (c) at (0,0) {};
 \node[int] (d) at (0,1.1) {};
 \draw (a) edge node[below] {$\scriptstyle e_3$} (c);
 \draw (b) edge node[below] {$\scriptstyle e_2$} (c);
 \draw (c) edge[bend left=15] node[right] {$\scriptstyle e_1$} (d);
\end{tikzpicture}\quad\quad
\begin{tikzpicture}[baseline=1ex]
 \node[int] (a) at (-1,-.5) {};
 \node[int] (b) at (1,-.5) {};
 \node[int] (c) at (0,0) {};
 \node[int] (d) at (0,1.1) {};
 \draw (a) edge[-latex] (b);
 \draw (a) edge[-latex] (c);
 \draw (a) edge[dotted] (d);
 \draw (b) edge[very thick] (c);
 \draw (b) edge[dotted] (d);
 \draw (c) edge[very thick,bend left=15] (d);
 \draw (c) edge[dotted,bend right=15] (d);
\end{tikzpicture}
$$
\caption{\label{fig:complexSigma}
For a skeleton $\Sigma$ on the left with chosen edges in the middle, an example of a graph in $\Sigma^2$ is drawn on the left.}
\end{figure}
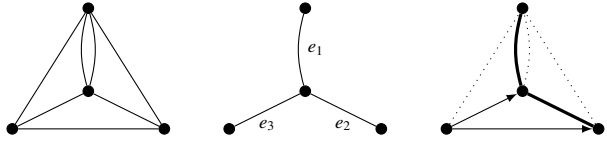

We define maps $f^{i+1}\colon\Sigma^i\rightarrow\Sigma^{i+1}$ which acts on the positional edge $e_{i+1}$ as follows
\begin{equation}
\label{def:s}
f^{i+1}:\,
\begin{tikzpicture}[baseline=-.65ex]
 \node[int] (a) at (0,0) {};
 \node[int] (b) at (1,0) {};
 \draw (a) edge[very thick,dotted] node[above] {$\scriptstyle m$} (b);
\end{tikzpicture}
\,\mapsto 0,\quad
\begin{tikzpicture}[baseline=-.65ex]
 \node[int] (a) at (0,0) {};
 \node[int] (b) at (1,0) {};
 \draw (a) edge[very thick,dotted, bend left=20] node[above] {$\scriptstyle m-1$} (b);
 \draw (a) edge[-latex, bend right=20] (b);
\end{tikzpicture}
\,\mapsto\,
\begin{tikzpicture}[baseline=-.65ex]
 \node[int] (a) at (0,0) {};
 \node[int] (b) at (1,0) {};
 \draw (a) edge[very thick,dotted, bend left=20] node[above] {$\scriptstyle m-1$} (b);
 \draw (a) edge[very thick, bend right=20] (b);
\end{tikzpicture}
\,,\quad
\begin{tikzpicture}[baseline=-.65ex]
 \node[int] (a) at (0,0) {};
 \node[int] (b) at (1,0) {};
 \draw (a) edge[very thick,dotted, bend left=20] node[above] {$\scriptstyle m-1$} (b);
 \draw (a) edge[latex-, bend right=20] (b);
\end{tikzpicture}
\,\mapsto -(-1)^n\,
\begin{tikzpicture}[baseline=-.65ex]
 \node[int] (a) at (0,0) {};
 \node[int] (b) at (1,0) {};
 \draw (a) edge[very thick,dotted, bend left=20] node[above] {$\scriptstyle m-1$} (b);
 \draw (a) edge[very thick, bend right=20] (b);
\end{tikzpicture}\,.
\end{equation}
The preferred direction (left to right) is the one away from the vertex $y$ along the thick tree.
It is easily seen that $f^{i+1}$ is a map of complexes. Therefore we have defined the chain of complexes and maps between them:
$$
\bar\mV\mO_{k+1}\sGC_{n+1}^\Sigma=\Sigma^0\,\tto{f^1}\,\Sigma^1\,\tto{f^2}\,\dots\,\tto{f^{v-1}}\,\Sigma^{v-1}.
$$

\begin{lemma}
\label{lem:s}
For every $i=0,\dots,v-2$ the map $f^{i+1}\colon\Sigma^i\rightarrow\Sigma^{i+1}$ is a quasi-isomorphism.
\end{lemma}
\begin{proof}
We set up a spectral sequence on the mapping cone of $f^{i+1}:\Sigma^i\rightarrow\Sigma^{i+1}$ on the number of dotted edges that are not in the positional edge $e_{i+1}$. The spectral sequence clearly converges correctly. On the first page there is a mapping cone of the map
$$
f^{i+1}:\left(\Sigma^i,\delta_0^{i+1}\right)\rightarrow\left(\Sigma^{i+1},0\right)
$$
where $\delta_0^{i+1}$ is the part of the differential acting on the positional edge $e_{i+1}$. This map is clearly quasi-isomorphism, concluding the proof.
\end{proof}

\begin{rem}
The good choice of maps $f^{i+1}$ would not be possible for multiple edges in $\mO_{k+1}\sGC_{n+1}^\mul$ for $n$ even. This is the very reason why did we get rid of those multiple edges in Subsection \ref{ss:tadmul}.
\end{rem}

$\Sigma^{v-1}$ is generated by only one graph, the one without solid edges. Let us call this graph $\bar\Sigma$. So $H(\Sigma^{v-1},\delta_0)$ is one-dimensional.

Let $f:=f^{v-1}\circ\dots\circ f^1$ and let us consider the composition $f\circ h:(\K\Sigma,0)\rightarrow\left(\Sigma^{v-1}\right)$. Recall \eqref{def:h} that 
$$
h(\Sigma)=
\sum_{\tau\in S(\Sigma)}\sum_{x\in V(\Sigma)}(v(x)-2)h_{x,\tau}(\Sigma).
$$
Since $f^i$ sends dotted edges to $0$, the term of the first sum that can survive the action of the composition $f$ is only the one where the chosen sub-tree $\tau$ coincides with the positional sub-tree $T:=\{e_1,\dots,e_{v-1}\}$. So
\begin{equation}
\label{sh}
f\circ h(\Sigma)=M\sum_{x\in V(\Sigma)}\pm(v(x)-2)\bar\Sigma,
\end{equation}
where $M$ is a positive factor coming from multiple edges in $T$ since, technically, if there are multiple edges in $T$ there are more ways to choose a sub-tree of real edges.

The goal is to show that $\Sigma$ is not sent to $0$. To do that, we need to be careful with the signs. We define the sign being $+$ if $x=y$. Let us chose $x$ distant from $y$ by $r$ edges in the tree $T$. We will show that the signs of this term in \eqref{sh} is also $+$.

Recall the sign convention in the definition of $h$ in Subsection \ref{ss:map}. For $n$ even pre-factor $(-1)^{nr}$ is always $+$. Vertices are indistinguishable, so the sign does not depend on the choice of the vertex $x$. But choosing another $x$ will change the convention of naming vertices in $h_{x,T}(\Sigma)$. This will give us a pre-factor $(-1)^r$ as described in Figure \ref{fig:signheven}. While acting by $s$, $r$ edges will go in opposite direction and give a pre-factor $(-(-1)^n)^r=(-1)^r$, recall \eqref{def:s}. The overall sign of $f\circ h_{x,T}$ is therefore $(-1)^r(-1)^r=+$.

\begin{figure}[h]
$$
\begin{tikzpicture}[baseline=-3ex]
 \node[int] (a) at (0,0) {};
 \node[int] (b) at (1,0) {};
 \node[int] (c) at (1.5,-.85) {};
 \node[int] (d) at (2.5,-.85) {};
 \draw (a) edge node[below] {$1$} (b);
 \draw (b) edge node[left] {$2$} (c);
 \draw (c) edge node[below] {$3$} (d);
 \draw (a) edge (-.5,.85) {};
 \draw (a) edge (-.5,-.85) {};
 \draw (b) edge (1.5,.85) {};
 \draw (c) edge (1,-1.7) {};
 \draw (d) edge (3,0) {};
 \draw (d) edge (3,-1.7) {};
 \node[left] at (a) {$y=x$};
\end{tikzpicture}
\quad\mxto{h}\quad
\begin{tikzpicture}[baseline=-3ex]
 \node[int] (a) at (0,0) {};
 \node[int] (b) at (1,0) {};
 \node[int] (c) at (1.5,-.85) {};
 \node[int] (d) at (2.5,-.85) {};
 \draw (a) edge[-latex] (b);
 \draw (b) edge[-latex] (c);
 \draw (c) edge[-latex] (d);
 \draw (a) edge[-latex] (-.5,.85) {};
 \draw (a) edge[-latex] (-.5,-.85) {};
 \draw (b) edge[-latex] (1.5,.85) {};
 \draw (c) edge[-latex] (1,-1.7) {};
 \draw (d) edge[-latex] (3,0) {};
 \draw (d) edge[-latex] (3,-1.7) {};
 \node[left] at (a) {$0$};
 \node[right] at (b) {$1$};
 \node[left] at (c) {$2$};
 \node[right] at (d) {$3$};
\end{tikzpicture}
$$
$$
\begin{tikzpicture}[baseline=-3ex]
 \node[int] (a) at (0,0) {};
 \node[int] (b) at (1,0) {};
 \node[int] (c) at (1.5,-.85) {};
 \node[int] (d) at (2.5,-.85) {};
 \draw (a) edge node[below] {$1$} (b);
 \draw (b) edge node[left] {$2$} (c);
 \draw (c) edge node[below] {$3$} (d);
 \draw (a) edge (-.5,.85) {};
 \draw (a) edge (-.5,-.85) {};
 \draw (b) edge (1.5,.85) {};
 \draw (c) edge (1,-1.7) {};
 \draw (d) edge (3,0) {};
 \draw (d) edge (3,-1.7) {};
 \node[left] at (a) {$y$};
 \node[right] at (d) {$x$};
\end{tikzpicture}
\quad\mxto{h}\quad
\begin{tikzpicture}[baseline=-3ex]
 \node[int] (a) at (0,0) {};
 \node[int] (b) at (1,0) {};
 \node[int] (c) at (1.5,-.85) {};
 \node[int] (d) at (2.5,-.85) {};
 \draw (a) edge[latex-] (b);
 \draw (b) edge[latex-] (c);
 \draw (c) edge[latex-] (d);
 \draw (a) edge[-latex] (-.5,.85) {};
 \draw (a) edge[-latex] (-.5,-.85) {};
 \draw (b) edge[-latex] (1.5,.85) {};
 \draw (c) edge[-latex] (1,-1.7) {};
 \draw (d) edge[-latex] (3,0) {};
 \draw (d) edge[-latex] (3,-1.7) {};
 \node[left] at (a) {$1$};
 \node[right] at (b) {$2$};
 \node[left] at (c) {$3$};
 \node[right] at (d) {$0$};
\end{tikzpicture}
$$
\caption{\label{fig:signheven}
For $n$ even edges in $\mO_k\GC_n$ and vertices in $\mO_{k+1}\sGC_{n+1}$ are odd. Choosing vertex $x$ that is $r$ edges away from $y$ along the tree shifts labels of vertices between $x$ and $y$. While renaming them to the state of $x=y$ we get the sign $(-1)^r$.}
\end{figure}
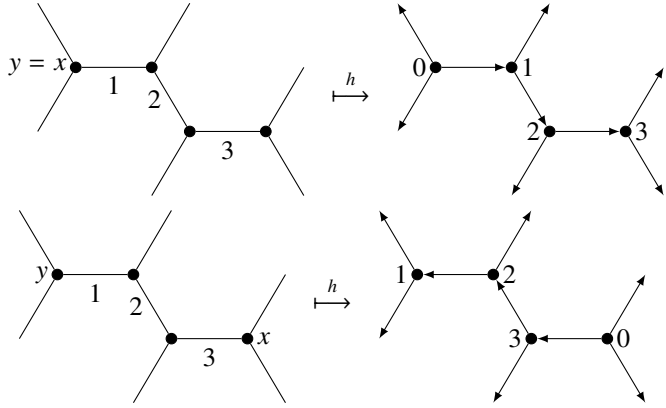

For $n$ odd exactly $r$ edges from the tree will change direction while acting by $h_{x,T}$, so there is a pre-factor $(-1)^r$ from the definition. If we choose another vortex $x$, we first need to rename it to $0$. We do the renaming such that the names of the edges at the end are the same, giving a pre-factor $(-1)^r$ as described in Figure \ref{fig:signhodd}. The overall sign of $f\circ h_{x,T}$ is $(-1)^r(-1)^r=+$. Acting by $f$ does not change sign in this case, recall \eqref{def:s}, so the overall sign of $f\circ h_{x,T}$ remains $+$.

\begin{figure}[h]
$$
\begin{tikzpicture}[baseline=-3ex]
 \node[int] (a) at (0,0) {};
 \node[int] (b) at (1,0) {};
 \node[int] (c) at (1.5,-.85) {};
 \node[int] (d) at (2.5,-.85) {};
 \draw (a) edge[->] (b);
 \draw (b) edge[->] (c);
 \draw (c) edge[->] (d);
 \draw (a) edge[->] (-.5,.85) {};
 \draw (a) edge[->] (-.5,-.85) {};
 \draw (b) edge[->] (1.5,.85) {};
 \draw (c) edge[->] (1,-1.7) {};
 \draw (d) edge[->] (3,0) {};
 \draw (d) edge[->] (3,-1.7) {};
 \node[left] at (a) {$0=y=x$};
 \node[right] at (b) {$1$};
 \node[left] at (c) {$2$};
 \node[right] at (d) {$3$};
\end{tikzpicture}
\quad\mxto{h}\quad
\begin{tikzpicture}[baseline=-3ex]
 \node[int] (a) at (0,0) {};
 \node[int] (b) at (1,0) {};
 \node[int] (c) at (1.5,-.85) {};
 \node[int] (d) at (2.5,-.85) {};
 \draw (a) edge[-latex] node[below] {$1$} (b);
 \draw (b) edge[-latex] node[left] {$2$} (c);
 \draw (c) edge[-latex] node[below] {$3$} (d);
 \draw (a) edge[-latex] (-.5,.85) {};
 \draw (a) edge[-latex] (-.5,-.85) {};
 \draw (b) edge[-latex] (1.5,.85) {};
 \draw (c) edge[-latex] (1,-1.7) {};
 \draw (d) edge[-latex] (3,0) {};
 \draw (d) edge[-latex] (3,-1.7) {};
 \node[left] at (a) {$0$};
\end{tikzpicture}
$$
$$
(-1)^r(-1)^r\quad
\begin{tikzpicture}[baseline=-3ex]
 \node[int] (a) at (0,0) {};
 \node[int] (b) at (1,0) {};
 \node[int] (c) at (1.5,-.85) {};
 \node[int] (d) at (2.5,-.85) {};
 \draw (a) edge[->] (b);
 \draw (b) edge[->] (c);
 \draw (c) edge[->] (d);
 \draw (a) edge[->] (-.5,.85) {};
 \draw (a) edge[->] (-.5,-.85) {};
 \draw (b) edge[->] (1.5,.85) {};
 \draw (c) edge[->] (1,-1.7) {};
 \draw (d) edge[->] (3,0) {};
 \draw (d) edge[->] (3,-1.7) {};
 \node[left] at (a) {$1=y$};
 \node[right] at (b) {$2$};
 \node[left] at (c) {$3$};
 \node[right] at (d) {$x=0$};
\end{tikzpicture}
\quad\mxto{h}\quad(-1)^r\quad
\begin{tikzpicture}[baseline=-3ex]
 \node[int] (a) at (0,0) {};
 \node[int] (b) at (1,0) {};
 \node[int] (c) at (1.5,-.85) {};
 \node[int] (d) at (2.5,-.85) {};
 \draw (a) edge[-latex] node[below] {$1$} (b);
 \draw (b) edge[-latex] node[left] {$2$} (c);
 \draw (c) edge[-latex] node[below] {$3$} (d);
 \draw (a) edge[-latex] (-.5,.85) {};
 \draw (a) edge[-latex] (-.5,-.85) {};
 \draw (b) edge[-latex] (1.5,.85) {};
 \draw (c) edge[-latex] (1,-1.7) {};
 \draw (d) edge[-latex] (3,0) {};
 \draw (d) edge[-latex] (3,-1.7) {};
 \node[right] at (d) {$0$};
\end{tikzpicture}
$$
\caption{\label{fig:signhodd}
For $n$ odd vertices in $\mO_k\GC_n$ and edges in $\mO_{k+1}\sGC_{n+1}$ are odd. If we choose vertex $x$ that is $r$ edges away from $y$ we need to rename it to $0$ and we shift all names of the vertices in between. This gives the sign $(-1)^r$. Map $h_{x,T}$ changes the direction of $r$ edges, so it gives another sign $(-1)^r$.}
\end{figure}
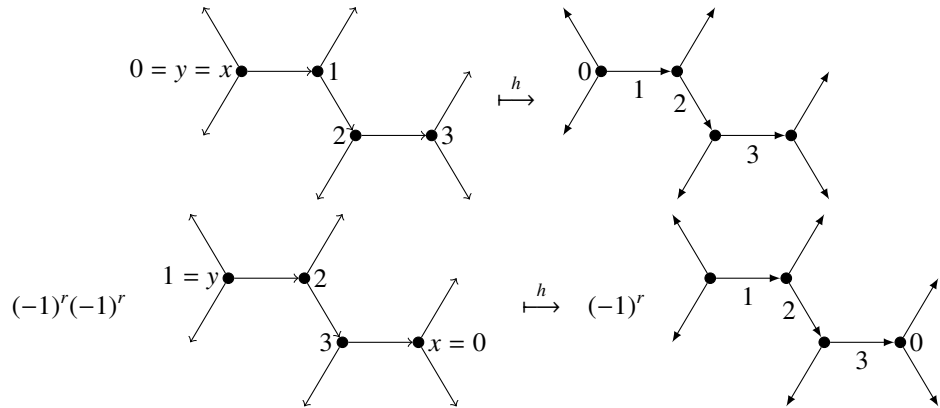

We have shown that $f\circ h$ maps the generator of the homology class in $(\K\Sigma,0)$ to the generator of the homology class in $(\Sigma^{v-1},\delta_0)$, so it is a quasi-isomorphism. Lemma \ref{lem:s} implies that $f$ is quasi-isomorphism, so $h$ from \eqref{core} is a quasi-isomorphism too. That was to be demonstrated.
\end{proof}

Theorem \ref{thm:main} now follows directly using Propositions \ref{prop:3}, \ref{prop:notadp} and \ref{prop:nomul}.

\end{document}